\documentclass[11pt]{amsart}

\usepackage{latexsym,amssymb,amsmath,graphics}%,showkeys}
\usepackage[dvips]{graphicx}

\def\ff{{\mathcal F}}
\def\ffi{\varphi}

\def\dst{\displaystyle}

\def\supp{{\mathrm{supp}\,}}

\def\C{{\mathbb{C}}}

\def\P{{\mathbb{P}}}

\def\R{{\mathbb{R}}}
\def\S{{\mathbb{S}}}

\def\Z{{\mathbb{Z}}}

\newcommand{\norm}[1]{{\left\|{#1}\right\|}}
\newcommand{\ent}[1]{{\left[{#1}\right]}}
\newcommand{\abs}[1]{{\left|{#1}\right|}}
\newcommand{\scal}[1]{{\left\langle{#1}\right\rangle}}
\newcommand{\set}[1]{{\left\{{#1}\right\}}}

\newenvironment{remark}[1][]{\vskip1pt\noindent\rm\textit{Remark #1}\,:\ }{\rm\vskip1pt}
\newenvironment{definition}[1][]{\vskip3pt\noindent\sl\textbf{Definition.}\ }{\rm\vskip3pt}

\newtheorem{problem}{Problem}

\newtheorem{lemma}{Lemma}[section]
\newtheorem{proposition}[lemma]{Proposition}
\newtheorem{theorem}[lemma]{Theorem}
\newtheorem{corollary}[lemma]{Corollary}
\newtheorem{remarknum}[lemma]{Remark}

\begin{document}

\title{On uncertainty principles in the finite dimensional setting}
%[Spartsity in two orthonormal bases]{On the sparsity of the representation of a finite dimensional 
%vector in two orthonormal bases}
\author{Saifallah Ghobber}

\address{S.G. D\'epartement Math\'ematiques\\
Facult\'e des Sciences de Tunis\\
Universit\'e de Tunis El Manar\\
Campus Universitaire\\ 1060 Tunis\\
Tunisie}
\email{Saifallah.Ghobber@math.cnrs.fr}

\author{Philippe Jaming}

\address{P.J. and S.G Universit\'e d'Orl\'eans\\
Facult\'e des Sciences\\ 
MAPMO - F\'ed\'eration Denis Poisson\\ BP 6759\\ F 45067 Orl\'eans Cedex 2\\
France}

\address{P.J. (current address)\,: Institut de Math\'ematiques de Bordeaux UMR 5251,
Universit\'e Bordeaux 1, cours de la Lib\'eration, F 33405 Talence cedex, France}
\email{Philippe.Jaming@gmail.com}

\begin{abstract}
The aim of this paper is to prove an uncertainty principle for the representation of a vector in two bases. Our result extends previously known ``qualitative'' uncertainty principles into more quantitative estimates.
We then show how to transfer such results to the discrete version of the Short Time Fourier Transform.
\end{abstract}



\maketitle

%\tableofcontents

\section{Introduction}

The aim of this paper is to deal with uncertainty principles in finite dimensional settings. Usually, an uncertainty principle says that a function
and its Fourier transform can not be both well concentrated.
Of course, one needs to give a precise meaning to ``well concentrated'' and
we refer to \cite{HJ,FS} for numerous versions of the uncertainty principle
for the Fourier transform in various settings. Our aim here is to present results of that flavour for unitary operators on $\C^d$ and then to apply those results to the discrete short-time Fourier transform.
Let us now be more precise and describe our main results and the relations with existing literature.

\subsection{Main results}
Before presenting our results, we need some further notation. Let $d$ be an integer and $\ell^2_d$ be $\C^d$ equipped with its standard norm denoted $\norm{a}_{\ell^2}$ or simply $\norm{a}_2$
and the associated scalar product $\scal{\cdot,\cdot}$. More generally, for
$0< p<+\infty$, the $\ell^p$-``norm'' is defined by
$\norm{a}_{\ell^p}=\left(\sum_{j=0}^{d-1}|a_j|^p\right)^{1/p}$.
For a set $E\subset\{0,\ldots,d-1\}$ we will write $E^c$ for its complement, $|E|$ for the number of its elements. Further, for $a=(a_0,\ldots,a_{d-1})\in\ell^2_d$, we denote
$\norm{a}_{\ell^2(E)}=\left(\sum_{j\in E}|a_j|^2\right)^{\frac{1}{2}}$.  Finally, the support of $a$
is defined as $\mbox{supp}\,a=\{j\,:\ a_j\not=0\}$ and we set $\norm{a}_{\ell^0}=|\mbox{supp}\,a|$.

Our aim here is to deal with finite dimensional analogues of the uncertainty principle. Here, instead of the Fourier transform, we will consider general unitary operators, ({\it i.e.} a change of coordinates 
from one orthonormal bases to another one), and concentration is measured in the following sense:

\begin{definition}\ \\
Let $\Phi=\{\Phi_j\}_{j=0,\ldots,d-1}$ and $\Psi=\{\Psi_j\}_{j=0,\ldots,d-1}$ be two orthonormal bases of $\ell^2_d$.
Let  $S,\Sigma\subset\{0,\ldots,d-1\}$.
Then $(S,\Sigma)$ is said to be a

--- \emph{weak annihilating pair} (for those bases) if
$\mathrm{supp}\,(\scal{a,\Phi_j})_{0\leq j\leq d-1}\subset S$ and
$\mathrm{supp}\,\scal{a,\Psi_j}\subset\Sigma$ implies that $a=0$;

--- \emph{strong annihilating pair} (for those bases) if there exists a constant
$C(S,\Sigma)$ such that for every $a\in\ell^2_d$
\begin{equation}
\label{eq:strongan}
\norm{a}_{\ell^2}\leq C(S,\Sigma)\bigl(\norm{\scal{a,\Phi_j}}_{\ell^2(S^c)}+
\norm{\scal{a,\Psi_j}}_{\ell^2(\Sigma^c)}\bigr).
\end{equation}
\end{definition}

Of course, any strong annihilating pair is also a weak one. This notion is an adaptation of a similar one
for the Fourier transform for which it has been extensively studied. We refer to \cite{HJ,FS} for more references.
The advantage of the second notion over the first one is that it states that
if the coordinates of $a$ in the basis $\Phi$ outside $S$ and those of $a$ in the basis $\Psi$ outside $\Sigma$ are small, then $a$ itself is small.

%In the opposit way,
%assume that $(S,\Sigma)$ is a strong annihilating pair. 
%Note that it is enough to prove \eqref{eq:strongan} when $\norm{a}_2=1$.
%But in $\C^n$, $\S^{n-1}:=\{a\in\ell^2_n\,:\ \norm{a}_2=1\}$ is compact and
%$\ffi\,:a\to \norm{a}_{\ell^2(S^c)}+\norm{Ta}_{\ell^2(\Sigma^c)}$ is 
%continuous, so that there exists $a_0\in\S^{n-1}$ that minimizes $\ffi$. If 
%this minimum 
%$m(S,\Sigma)$ where $0$, then $\mathrm{supp}\,a\subset S$ and
%$\mathrm{supp}\,Ta\subset\Sigma$ which implies that $a=0$ and contradicts $\norm{a}=1$. Thus $m(S,\Sigma)\not=0$ and $C(S,\Sigma)=1/m(S,\Sigma)$.

It follows from a standard compactness argument (that we reproduce after Formula \eqref{eq:strtriv} below) that, in a finite dimensional setting, both notions are equivalent. However, this argument
does not give any information on $C(S,\Sigma)$. It is our
aim here to modify an argument from \cite{DS} to obtain quantitative information on this constant in terms of $S$ and $\Sigma$. More precisely, we will prove the following Uncertainty Principles:

\medskip

\noindent{\bf Theorem A.}\ \\
{\sl Let $d$ be an integer. Let $\Phi=\{\Phi_j\}_{j=0,\ldots,d-1}$ and $\Psi=\{\Psi_j\}_{j=0,\ldots,d-1}$ be two orthonormal bases of $\C^d$ and define the \emph{coherence} of $\Phi,\Psi$ by
$M(\Phi,\Psi)=\dst\max_{0\leq j,k\leq d-1}|\scal{\Phi_j,\Psi_k}|$.
Let $S,\Sigma$ be two subsets of $\{0,\ldots,d-1\}$. Assume that $|S||\Sigma|<\dst\frac{1}{M(\Phi,\Psi)^2}$.
Then for every $a\in\C^d$,}
$$
\|a\|_2\leq\left(1+\frac{1}{1-M(\Phi,\Psi)(|S||\Sigma|)^{1/2}}\right)\left(
\norm{\scal{a,\Phi_j}}_{\ell^2(S^c)}+\norm{\scal{a,\Psi_j}}_{\ell^2(\Sigma^c)}\right).
$$

\medskip

As a first corollary of this result, we will show that a sequence may not be too compressive
in two different bases ({\it see} Corollary \ref{cor:comp} for a precise statement).
Further, we show in Proposition \ref{prop:BT} that, 
if $M(\Phi,\Psi)=d^{-1/2}$ (in which case, the bases are said to be \emph{unbiased}),
then any set $\Sigma$ that is not too large is a member of a strong annihilating pair. More precisely, if
$|\Sigma|\leq d-\sqrt{240d}$, there exists a set $S$ such that
$|S|\geq\dst\frac{(d-|\Sigma|)^2}{240d}$ and $(S,\Sigma)$ is a strong annihilating pair.

Let us stress that all results mentioned so far apply to the discrete Fourier transform $\ff_d$ which may be seen as the
unitary operator that changes the standard basis $\Delta=\{\delta_j\}_{j=0,\ldots,d-1}$ of $\ell^2_d$
into the \emph{Fourier basis} defined by $\Psi=\{\Psi_j\}_{j=0,\ldots,d-1}$ with 
$$
\Psi_j=\ff_d[\delta_j]=d^{-1/2}(1,\ldots,e^{2i\pi jk/d},\ldots,e^{2i\pi j(d-1)/d}).
$$
Note that these two bases are unbiased.

\medskip

Finally, we will apply our results to the discrete short-time Fourier transform. Let us describe these results in a 
slightly simplified setting. First, for $d$ an integer, we will consider elements of $\ell^2_d$ as $d$-periodic 
functions on $\{0,\ldots,d-1\}$. For $f,g\in\ell^2_d$, the short-time (or windowed) Fourier
transform of $f$ with window $g$ is then defined for $j,k\in\{0,\ldots,d-1\}$ by
$$
V_gf(j,k)=\frac{1}{\sqrt{d}}\sum_{\ell=0}^{d-1}f(\ell)\overline{g(\ell-j)}e^{2i\pi k\ell/d}.
$$
Note that, if we write $\tau_jg(\ell)=g(\ell-j)$, then $V_gf(j,k)=\ff_d[f\,\overline{\tau_jg}](k)$, so the 
windowed Fourier transform can be seen as the Fourier transform of $f$ seen through a sliding window $g$.
We refer to {\it e.g.} \cite{HCM,HS,KPR} for various applications of the discrete short-time Fourier transform in signal processing. Our aim is to show that this transform satisfies an uncertainty principle:

\medskip

\noindent{\bf Theorem B.}\ \\
{\sl Let $\Sigma$ be a subset of $\{0,\ldots,d-1\}^2$ with $|\Sigma|<d$ and $g\in\ell^2_d$ with $\norm{g}_2=1$. 
Then for every $f\in\ell^2_d$,}
$$
\norm{f}_2\leq\frac{2\sqrt{2}}{1-|\Sigma|/d}\left(\sum_{(j,k)\notin\Sigma}|V_gf(j,k)|^2\right)^{1/2}.
$$

\medskip

Following \cite{KPR}, the definition of the windowed Fourier transform will be extended to the setting of finite Abelian groups.
We will prove an analogue of the above theorem in that general setting. 

\subsection{Comparison with existing results}.

The two uncertainty principles given in Theorems A and B are quantitative
improvements of known results. 

\medskip

First, uncertainty principles for the discrete Fourier transform $\ff_d$ are known 
for some time. To our knowledge, the first occurrence of such a result is due to Matolcsi-Szucs\cite{MS}
and was rediscovered by Donoho-Stark \cite{DS}. More precisely, if one considers $\ff_d$
as the change of coordinate operator from the standard basis to the Fourier basis then
%, if $\Phi$ is the standard basis and $\Psi$ is the so-called \emph{Fourier basis}
%given by $\Psi_j=\ff_d\Phi_j=(d^{-1/2}e^{2i\pi jk/d})_{k=0,\ldots,d-1}$, then $M(\Phi,\Psi)=d^{-1/2}$.
%So, 
Theorem A reads as follows: if $|S||\Sigma|<d$ then
\begin{equation}
\label{eq:UP}
\|a\|_2\leq\left(1+\frac{1}{1-(|S||\Sigma|/d)^{1/2}}\right)\left(
\norm{a}_{\ell^2(S^c)}+\norm{\ff_d[a]}_{\ell^2(\Sigma^c)}\right).
\end{equation}
In particular, if $a$ is supported in $S$ and $\ff_d[a]$ is supported in $\Sigma$, then $a=0$,
which is the result proved in \cite{MS,DS}.

This result may also be seen as a discrete counterpart of an uncertainty principle for the continuous
Fourier transform on $L^2(\R^n)$ originally proved by F. Nazarov for $n=1$ \cite{Na} and
the second author for arbitrary dimension \cite{Janaz}. This was one of the motivations in writing this paper.

At this stage, we would also like to mention that T. Tao \cite{Ta} proved that, if the dimension $d$
is a prime number, and if $|S|+|\Sigma|\leq d$, than $(S,\Sigma)$ is an annihilating pair. Unfortunately,
Tao's proof does not give any information on the constant $C(S,\Sigma)$, and our method
of proof does not recover his result neither. For sake of completeness, we would like to
mention the work of Meshulam \cite{Me} and Delvaux-Van Barel \cite{DV1,DV2} that pursue Tao's work.

Further, Donoho-Huo \cite{DH} considered other particular pairs of bases. For the general case of two arbitrary
bases, Theorem \ref{th:th1intro} gives a quantitative version of a result of 
Elad-Bruckstein \cite{EB}. Note that Elad-Bruckstein's result was extended to more than two bases by 
Gribonval-Nielsen \cite{GN}. Note also that \cite{DH,EB,GN} further deal with the problem of recovering a
vector that is sparse ({\it i.e.} with small support) in one basis from knowledge of a small number of its coordinates in an other basis via $\ell^1$-minimization, an issue we do not tackle here. 

Next, if the sets $S,\Sigma$ are chosen randomly, then one may improve the result in Theorem A.
This was done for the discrete Fourier transform by Cand\`es-Tao and almost simultaneously by
Rudelson-Vershynin \cite{RV} who obtained a slightly better result that also applies to unbiased bases. As we will use it for the discrete short-time Fourier transform, we will reproduce the result in Theorem \ref{th:RV}.
Further results of probabilistic nature may be found in the work of Tropp \cite{Tr,Tr0}.

\medskip

Finally, the uncertainty principle for the short-time Fourier transform that we prove here is a quantitative strengthening  of the main result of Krahmer, Pfander, Rashkov \cite{KPR}. Its proof is an adaptation of
a method that was originally developed in \cite{Jam,Jan} and improved in \cite{De,GZ} in the continuous setting.
More precisely, we first prove that the discrete Fourier transform of the product of two short-time Fourier 
transforms is again a product of short-time Fourier transforms (Lemma \ref{lem:sym}). This allows us to prove
a transfer principle from strong annihilating pairs for the discrete 
Fourier transform into a similar result for its short-time version (Lemma \ref{lem:transfer}).
From this, we deduce Theorem B (in a more general version, Corollary \ref{cor:cor2intro})
as well as a ``probabilistic improvement'' (Corollary \ref{cor:RVamb}).

\subsection{Link with compressive sensing}
Although our results do not apply directly to the blooming subject of compressed sensing, this subject was one of the motivations of our research. Let us recall that the \emph{Uniform Uncertainty Principle} was introduced by E. Cand\`es and T. Tao in their seminal series of papers \cite{CT1,CT2,CRT,CR}.

\begin{definition}\ \\
Let $T\,:\ell^2_d\to\ell^2_d$ be a unitary operator. Let $s\leq d$ be an integer
and $\Omega\subset\{0,\ldots,d-1\}$.
% and $P_\Omega$ be the projection 
%$P_\Omega a=(\chi_\Omega(j)a_j)_{0\leq j\leq d-1}$.
Then $(T,\Omega,s)$ is said to have the \emph{Uniform Uncertainty Principle}
(also called the \emph{Restricted Isometry Property})
if there exists $\delta_s\in(0,1)$ such that, for every $S\subset\{0,\ldots,d-1\}$ with $|S|=s$ and for every $a\in\ell^2_d$
with $\mathrm{supp}\,a\subset S$
\begin{equation}
\label{eq:uup}
(1-\delta_s)\frac{|\Omega|}{d}\norm{a}_2^2\leq\norm{Ta}_{\ell^2(\Omega)}^2\leq
(1+\delta_s)\frac{|\Omega|}{d}\norm{a}_2^2
\end{equation}
We will call $\delta_s$ the \emph{Restricted Isometry Constant} of $(T,\Omega,s)$.
\end{definition}

The purpose of this property was to show that, one may recover $a$ from the knowledge of $P_\Omega Ta$
(where $P_\Omega$ is the projection onto the coordinates in $\Omega$), under the restriction of $a$ to be sufficiently \emph{sparse}, that is $|\mbox{supp}\,a|$ to be sufficiently small.
Moreover, if $\delta_{2s}$ is sufficiently small, $a$ may be reconstructed by an $\ell^1$-minimization
program ({\it see} the paper of  Cand\`es \cite{Ca} and Foucard-Lai \cite{FL} for the best results to date).
%Let us mention how a recent result due to  adapts in our case:
%
%\begin{theorem}[Cand\`es \cite{Ca}]\label{th:candes}\ \\
%$T\,:\ell^2_d\to\ell^2_d$ be a unitary operator, let $s\leq\frac{d-1}{2}$ be an integer and
%and $\Omega\subset\{0,\ldots,d-1\}$ and $\eps>0$. Assume that $(T,\Omega,2s)$ satisfies the Uniform Uncertainty Principle with restricted isometry constant $\delta_{2s}$ satisfying $\delta_{2s}<\sqrt{2}-1=(\sqrt{2}+1)^{-1}$.
%
%Then for every $a\in\ell^2_d$ with $|\mbox{supp}\,a|\leq s$,
%for every $e\in\ell^2_d$ with $\norm{e}_2<\eps$, the solution $\underline{a}$ of
%$$
%\min\{\norm{\tilde a}_{\ell^1}\,:\ \tilde a\in\C^d,\ \norm{T\tilde a-\bigl(Ta+e\bigr)}_2\leq\eps\}
%$$
%satisfies
%$$
%\norm{\underline{a}-a}_2\leq 2\frac{1+(\sqrt{2}-1)\delta_{2s}}{{1-(\sqrt{2}+1)\delta_{2s}}}s^{-1/2}\norm{a-a_s^*}_2
%+\frac{4}{1-(\sqrt{2}+1)\delta_{2s}}\epsilon
%$$
%where $a_s^*$ is a vector that minimizes $\norm{a-a_s}_2$
%among all vectors $a_s$ such that $|\supp a_s|\leq s$.
%
%In particular, if $a$ has support of size $s$ (thus $a=a_s^*$) and $\eps=0$, then $\underline{a}=a$.
%\end{theorem}
%
%In \cite{Ca} the operator $T$ is \emph{real}, that is, it maps $\R^d\to\R^d$. In order to adapt the result to complex operators, a stronger condition on $\delta_{2s}$ needs to be imposed. On the other hand, we can take $\delta_s=0$ in \eqref{eq:rip2},
%which allows for some gain, so that we are back to the same condition as in \cite{Ca}.
%This follows from straight forward adaptations of the proofs in \cite{Ca}.

Let us now mention how the Uniform Uncertainty Principle (UUP) is linked to the
notion of annihilating pairs.
 If $(T,\Omega,s)$
has the UUP with constant $\delta_s$ then for every $S$ of cardinality $s$,
$(S,\Omega^c)$ is an annihilating pair for the standard basis $\Delta=\{\delta_j\}_{0\leq j\leq d-1}$
and the orthonormal basis $T\Delta=\{T\delta_j\}_{0\leq j\leq d-1}$. More precisely, 
a standard computation ({\it see} \eqref{eq:simple} where we reproduce the simple argument) shows that
$$
\norm{a}_2\leq\left(1+\sqrt{\frac{d}{(1-\delta_s)|\Omega|}}\right)\bigl(\norm{a}_{\ell^2(S^c)}
+\norm{Ta}_{\ell^2(\Omega)}\bigr).
%\leq\frac{2}{\sqrt{1-\delta_s}}\bigl(\norm{a}_{\ell^2(S^c)}
%+\norm{Ta}_{\ell^2(\Omega)}\bigr).
$$
%One of the main results of Cand\`es {\it et al} was to prove that, if $T$
%is the discrete Fourier transform and if $\Omega$
%has been chosen at random, then $\delta_s<1/2$ with high probability. This was further improved by M. Rudelson and R. Vershynin \cite{RV} and J. Tropp \cite{Tr}.
Conversely, assume that $\Sigma$ is such that, for every $S$
such that $|S|=s$, $(S,\Sigma)$ is a strong annihilating pair for $\Delta$ and $T\Delta$. Let $C(\Sigma)=\sup_{|S|=s}C(S,\Sigma)$, then $(T,\Sigma^c,s)$ satisfies the Uniform Uncertainty Principle
with $\dst\delta_s=1-\frac{1}{C(\Sigma)}\frac{1}{1-|\Sigma|/d}$.

\medskip

\subsection*{Outline of the paper}
This article is organized as follows: in the next section, we prove results about strong annihilating pairs for a change of basis. The following section deals with applications to the short-time Fourier transform. We devote the last section to a short conclusion.

\section{The Uncertainty Principle for expansions in two bases.}

\subsection{Further notations on Hilbert spaces}\ 

Let $\Phi=\{\Phi_j\}_{j=0,\ldots,d-1}$ be a basis of $\C^d$ that is \emph{normalized} {\it i.e.}
$\|\Phi_j\|_2=1$ for all $j$. If $a\in\C^d$, then we may write $a=\dst\sum_{i=0}^{d-1}a_i\Phi_i$.
We will denote by $\norm{a}_{\ell^p(\Phi)}=\norm{(a_0,\ldots,a_{d-1})}_{\ell^p}$
and $\mbox{supp}_\Phi\,a=\set{i\,: a_i\not=0}$. We also define 
$\norm{a}_{\ell^2(\Phi,E)}$ in the obvious way when $E$ is a subset of $\{0,\ldots,d-1\}$. When no confusion can arise, we simply write $\norm{a}_{\ell^2(E)}$. 

\medskip

Next, we will denote by $\Phi^*=\{\Phi_j^*\}_{j=0,\ldots,d-1}$ the dual basis\footnote{We would like to point the reader's attention to the notation adopted here and that is standard in linear algebra. The vectors $\Phi_j$ may be seen as the row (or column) vector of its coordinates in the standard basis. Then $\Phi_j^*$ may be seen as the transposed-conjugate of $\Phi_j$ if the basis $\Phi_j$ is orthonormal. In general, this is not the case.} of $\Phi$. More precisely,
$\Phi^*$ is the basis defined by
$\scal{\Phi_j,\Phi_k^*}=\delta_{j,k}$ where $\delta_{j,k}$ is the Kronecker symbol, $\delta_{j,k}=
\begin{cases}0&\mbox{if }j\not=k\\ 1&\mbox{if }j=k\end{cases}$.
Every $a\in\C^d$ can then be written as
$$
a=\sum_{j=0}^{d-1}\scal{a,\Phi_j^*}\Phi_j.
$$
Moreover, there exist two positive numbers $\alpha(\Phi)$ and $\beta(\Phi)$,
called the \emph{lower} and \emph{upper Riesz bounds} of $\Phi$ such that
\begin{equation}
\label{eq:frame}
\alpha(\Phi)\|a\|_2\leq\left(\sum_{j=0}^{d-1}|\scal{a,\Phi_j^*}|^2\right)^{1/2}\leq\beta(\Phi)\|a\|_2.
\end{equation}
Note that, if we take $a=\Phi_k$, we obtain $\alpha(\Phi)\leq 1 \leq\beta(\Phi)$. Moreover, $\alpha(\Phi)=\beta(\Phi)=1$ if and only if $\Phi$ is orthonormal.

If $\Phi$ and $\Psi$ are two normalized bases of $\C^d$, we will define their \emph{coherence} by
$$
M(\Phi,\Psi)=\max_{0\leq j,k\leq d-1}\abs{\scal{\Phi_j,\Psi_k}}.
$$
Obviously $M(\Phi,\Psi)\leq 1$ and, if $\Phi$ and $\Psi$ are orthonormal bases, then
$M(\Phi,\Psi)\geq\dst\frac{1}{\sqrt{d}}$.
If $M(\Phi,\Psi)=\frac{1}{\sqrt{d}}$, then $\Phi$ and $\Psi$ are said to be \emph{unbiased}.
A typical example of a pair of unbiased bases
is the standard basis and the Fourier basis of $\C^d$, {\it see} Section \ref{sec:abel}.

\medskip

Let us recall that the Hilbert-Schmidt norm of a linear operator is the $\ell^2_d$ norm
of its matrix in an orthonormal basis $\Phi$:
$$
\norm{U}_{HS}=\left(\sum_{i,j=0}^{d-1}|\scal{U\Phi_i,\Phi_j}|^2\right)^{\frac{1}{2}}.
$$
As is well known, this definition does not depend on the orthonormal basis and it controls the
norm of $U$ as a linear operator $\ell^2_d\to\ell^2_d$\,:
$$
\norm{U}_{\ell^2_d\to\ell^2_d}:=\max\limits_{a\in\C^d\,:\ \|a\|_2=1}\|Ua\|_2\leq\norm{U}_{HS}.
$$

\subsection{The strong version of Elad and Bruckstein's Uncertainty Principle.}\ 

Let us start by giving a simple proof of a result of Elad and Bruckstein \cite{EB}.

\begin{lemma}\label{lem:1}\ \\
Let $\Phi$ and $\Psi$ be two normalized bases of $\C^d$.
Then for every $a\in\C^d\setminus\{0\}$,
$$
\norm{a}_{\ell^0(\Phi)}\norm{a}_{\ell^0(\Psi)}\geq \frac{1}{\left(\min\left\{\frac{\beta(\Phi)}{\alpha(\Psi)}M(\Phi,\Psi^*),
\frac{\beta(\Psi)}{\alpha(\Phi)}M(\Phi^*,\Psi)\right\}\right)^2}.
$$
In particular,
$$
\norm{a}_{\ell^0(\Phi)}+\norm{a}_{\ell^0(\Psi)}\geq
\frac{2}{\min\left\{\frac{\beta(\Phi)}{\alpha(\Psi)}M(\Phi,\Psi^*),
\frac{\beta(\Psi)}{\alpha(\Phi)}M(\Phi^*,\Psi)\right\}}.
$$
\end{lemma}

\begin{proof} As the arithmetic mean dominates the geometric mean, the second statement immediately
follows from the first one. The proof mimics the proof given in \cite{Ta} for the Fourier basis.
For $a\not=0$ and $j=0,\ldots, d-1$,
\begin{eqnarray*}
|\scal{a,\Psi_j^*}|&=&\abs{\sum_{k=0}^{d-1}\scal{a,\Phi_k^*}\scal{\Phi_k,\Psi_j^*}}
\leq\left(\max_{j,k=0,\ldots,d-1}\abs{\scal{\Phi_k,\Psi_j^*}}\right)\sum_{k=0}^{d-1}\abs{\scal{a,\Phi_k^*}}\\
&\leq&M(\Phi,\Psi^*)|\mbox{supp}_\Phi\,a|^{1/2}
\left(\sum_{k=0}^{d-1}\abs{\scal{a,\Phi_k^*}}^2\right)^{1/2}\\
&\leq&\beta(\Phi)M(\Phi,\Psi^*)\norm{a}_{\ell^0(\Phi)}^{1/2}\,\|a\|_{\ell^2}\\
&\leq&\frac{\beta(\Phi)}{\alpha(\Psi)}M(\Phi,\Psi^*)\norm{a}_{\ell^0(\Phi)}^{1/2}
\left(\sum_{k=0}^{d-1}\abs{\scal{a,\Psi_k^*}}^2\right)^{1/2}\\
&\leq&\frac{\beta(\Phi)}{\alpha(\Psi)}M(\Phi,\Psi^*)
\norm{a}_{\ell^0(\Phi)}^{1/2}\norm{a}_{\ell^0(\Psi)}^{1/2}\max_{k=0,\ldots,d-1}|\scal{a,\Psi_k^*}|.
\end{eqnarray*}
It follows that
$$
\norm{a}_{\ell^0(\Phi)}\norm{a}_{\ell^0(\Psi)}\geq \left(\frac{\beta(\Phi)}{\alpha(\Psi)}M(\Phi,\Psi^*)\right)^{-2}.
$$
Exchanging the roles of $\Phi$ and $\Psi$, we obtain the result.
\end{proof}

\begin{remarknum}
Let $\Phi$ and $\Psi$ be two unbiased orthonormal bases. The lemma then reads
$|\mbox{supp}_\Phi\,a||\mbox{supp}_{\Psi}\,a|\geq d$. We can thus reformulate the lemma as follows:
if $S$ and $\Sigma$ are two subsets of $\{1,\ldots,d\}$ with $|S||\Sigma|<d$, and if
$\mbox{supp}_\Phi\,a\subset S$ and $\mbox{supp}_{\Psi}\,a\subset\Sigma$ then $a=0$.
\end{remarknum}

We will now switch to strong annihilating pairs. First, note that if $(S,\Sigma)$ is a weak annihilating pair,
then it
%In other words, if we denote by $\ff$ the (unitary) operator that is defined by $\ff\Phi_i=\Psi_i$, then
%$(S,\Sigma)$ is an annihilating pair for $\ff$. 
%When $\Phi$ is the standard basis and $\Psi$ the Fourier basis, this result stems back to 
%Matolcsi-Sz\"ucks \cite{MS} and Donoho-Stark \cite{DS}. 
%It follows that $(S,\Sigma)$ 
is also a \emph{strong annihilating pair}, {\it i.e.}
there exists a constant $C=C(S,\Sigma,\Phi,\Psi)$ such that, for every $a\in\C^d$,
\begin{equation}
\label{eq:strtriv}
\|a\|_2\leq C\Bigl(\|a\|_{\ell^2(\Phi,S^c)}+\|a\|_{\ell^2(\Psi,\Sigma^c)}\Bigr).
\end{equation}
Indeed, $C=D^{-1}$ where $D$ is the minimum of $\|a\|_{\ell^2(\Phi,S^c)}+\|a\|_{\ell^2(\Psi,\Sigma^c)}$
over $a\in\S^{d-1}$, the unit sphere of $\C^d$. This minimum is reached in some $a_0\in\S^{d-1}$ and is thus non-zero since $(S,\Sigma)$ is a
weak annihilating pair.
%
%Indeed, by homogeneity it is enough to prove \eqref{eq:strtriv} when $\norm{a}_2=1$. But
%the unit sphere $\S^d=\{a\in\C^d\,:\ \norm{a}_2=1\}$ of $\C^d$ is compact and the map
%$a\to\|a\|_{\ell^2(\Phi,S^c)}+\|a\|_{\ell^2(\Psi,\Sigma^c)}$ is continuous, thus its minimum $D$ over $\S^d$ is reached in some $a_0$. If this minimum were $0$, then $\mbox{supp}_\Phi\,a_0\subset S$
%and $\mbox{supp}_\Psi\,a_0\subset \Sigma$ thus $a_0=0$ which would contradict $\norm{a_0}_2=1$. Thus $D>0$ and \eqref{eq:strtriv} is satisfied with $C=D^{-1}$.
However, this does not allow to obtain an estimate on the constant $C$. This will be overcome in the next theorem,
Theorem A from the introduction.

\begin{theorem}\label{th:th1intro}\ \\
Let $d$ be an integer. Let $\Phi$ and $\Psi$ be two orthonormal bases of $\C^d$
and $S,\Sigma$ be two subsets of $\{0,\ldots,d-1\}$. Assume that $|S||\Sigma|<\dst\frac{1}{M(\Phi,\Psi)^2}$.
Then for every $a\in\C^d$,
$$
\|a\|_2\leq\left(1+\frac{1}{1-M(\Phi,\Psi)(|S||\Sigma|)^{1/2}}\right)\left(\norm{a}_{\ell^2(\Phi,S^c)}+\norm{a}_{\ell^2(\Psi,\Sigma^c)}\right).
$$
\end{theorem}

\begin{remark}\ \\
For comparison with the previous lemma, recall that as $\Phi,\Psi$ are orthonormal
they are equal to their dual bases and that their lower and upper Riesz bounds are $1$.
\end{remark}

\begin{proof} The proof we present here is in the spirit of \cite{HJ} and is also inspired by \cite{DS}.

 %Without loss of generality, we will assume that $\Phi$ is the standard basis.
Let $U$ be the change of basis from
$\Psi$ to $\Phi$, that is the linear operator defined by $U\Psi_i=\Phi_i$. We will still denote by $U$
its matrix in the basis $\Phi_i$, so that $U=[U_{i,j}]_{1\leq i,j\leq d}$ is given by
$U_{i,j}=\scal{\Phi_j,\Psi_i}$. As $U$ is unitary, $U^*\Phi_i=\Psi_i$.

For a set $E\subset\{1,\ldots,d\}$ let $P_E$ be the projection
$P_Ea=\sum_{j\in E}\scal{a,\Phi_j}\Phi_j$. A direct computation then shows that 
$\norm{a}_{\ell^2(\Phi,S^c)}=\|P_{S^c}a\|_2$ while
\begin{eqnarray}
\norm{a}_{\ell^2(\Psi,E)}&=&\left(\sum_{j\in E}\abs{\scal{a,\Psi_j}}^2\right)^{1/2}
=\left(\sum_{j\in E}\abs{\scal{a,U^*\Phi_j}}^2\right)^{1/2}\notag\\
&=&\left(\sum_{j\in E}\abs{\scal{Ua,\Phi_j}}^2\right)^{1/2}
=\left(\sum_{j=1}^d\abs{\scal{P_EUa,\Phi_j}}^2\right)^{1/2}\notag\\
&=&\|P_EUa\|_2.\notag%\label{eq:pea}
\end{eqnarray}

Assume first that $a\in\C^d$ is such that $\mbox{supp}_\Phi\,a\subset S$. Then
$$
\|P_{\Sigma}Ua\|_2=\|P_\Sigma UP_Sa\|_2
\leq \norm{P_\Sigma UP_S}_{\ell^2\to\ell^2}\norm{a}_{\ell^2(\Phi,S)}.
$$
It follows that
\begin{eqnarray}
\norm{a}_{\ell^2(\Psi,\Sigma^c)}&=&
\|P_{\Sigma^c}Ua\|_2\geq \|Ua\|_2-\|P_{\Sigma}Ua\|_2
\geq\|a\|_2-\norm{P_\Sigma UP_S}_{\ell^2\to\ell^2}\norm{a}_{\ell^2(\Phi,S)}\nonumber\\
&=&\Bigl(1-\norm{P_\Sigma UP_S}_{\ell^2\to\ell^2}\Bigr)\norm{a}_{\ell^2(\Phi,S)}.\label{eq:CS1}
\end{eqnarray}
The last equality comes from the assumption $\mbox{supp}_\Phi\,x\subset S$ which implies $\|a\|_2=\norm{a}_{\ell^2(\Phi,S)}$.

Note that, if we are able to prove that $\norm{P_\Sigma UP_S}_{\ell^2\to\ell^2}<1$, then this inequality
implies that $(S,\Sigma)$ is an annihilating pair.
The following computation allows to estimate the constant $C(S,\Sigma)$ appearing in the definition
of a strong annihilating pair: write 
$D=\Bigl(1-\norm{P_\Sigma UP_S}_{\ell^2\to\ell^2}\Bigr)^{-1}$
then for $a\in\C^d$,
\begin{eqnarray}
\|a\|_2&=&\|P_Sa+P_{S^c}a\|_2 \leq\|P_Sa\|_2+\|P_{S^c}a\|_2%\nonumber\\
%&\leq& 
\leq D\|P_{\Sigma^c}UP_Sa\|_2+\|P_{S^c}a\|_2\nonumber\\
&=&D\|P_{\Sigma^c}U(a-P_{S^c}a)\|_2+\|P_{S^c}a\|_2\nonumber\\
&\leq&D\|P_{\Sigma^c}Ua\|_2+D\|UP_{S^c}a\|_2+\|P_{S^c}a\|_2\label{eq:simple}
\end{eqnarray}
since $\|P_{\Sigma^c}x\|_2\leq\|x\|_2$ for every $x\in\C^d$. Now, as $U$ is unitary, we get
$$
\|a\|_2\leq D\|P_{\Sigma^c}Ua\|_2+\bigl(1+D\bigr)\|P_{S^c}a\|_2
$$
which immediately gives an estimate of the desired form with 
$$
C(S,\Sigma,\Phi,\Psi)=1+\Bigl(1-\norm{P_\Sigma UP_S}_{\ell^2\to\ell^2}\Bigr)^{-1}.
$$
It remains to give an upper bound on $\norm{P_\Sigma UP_S}_{\ell^2\to\ell^2}$:
\begin{eqnarray}
\norm{P_\Sigma UP_S}_{\ell^2\to\ell^2}&\leq&\norm{P_\Sigma UP_S}_{HS}
=\left(\sum_{i=1}^d\sum_{j=1}^d\abs{\scal{\Phi_i,P_\Sigma UP_S\Phi_j}}^2\right)^{1/2}\nonumber\\
&=&\left(\sum_{i\in\Sigma}\sum_{j\in S}\abs{\scal{\Phi_i,U\Phi_j}}^2\right)^{1/2}\nonumber\\
&\leq&M(\Phi,\Psi)(|S||\Sigma|)^{1/2}\label{eq:CS2}
\end{eqnarray}
which completes the proof of the theorem.
\end{proof}

\begin{remark}
A similar result can be obtained for more general bases.
% More precisely, {\sl let $\Phi,\Psi$ be two bases of $\C^d$ and $S,\Sigma$ two subsets of $\{0,\ldots,d-1\}$ such that
%$\dst|S||\Sigma|<\frac{\alpha(\Psi)^2}{\beta(\Phi)^2M(\Phi,\Psi^*)^2}$. Then}
%We may of course exchange $\Phi$ and $\Psi$ in the above inequality.
%
%\medskip
%
Let us outline the proof of such a result: let $\Phi,\Psi$ be two bases of $\C^d$ and $S,\Sigma$ two subsets of $\{0,\ldots,d-1\}$ such that
$|S||\Sigma|\beta(\Phi)^2M(\Phi,\Psi^*)^2<\alpha(\Psi)^2$.
 The following computation replaces \eqref{eq:CS1} and \eqref{eq:CS2}:
if $\supp_\Phi a\subset S$,
\begin{eqnarray*}
\sum_{j\in\Sigma}|\scal{a,\Psi_j^*}|^2
&=&\sum_{j\in\Sigma}\abs{\sum_{k\in S}\scal{a,\Phi_k^*}\scal{\Phi_k,\Psi_j^*}}^2%\\
%&\leq&
\leq\sum_{j\in\Sigma}\left(\sum_{k\in S}|\scal{a,\Phi_k^*}|^2\right)
\left(\sum_{k\in S}|\scal{\Phi_k,\Psi_j^*}|^2\right)\\
&\leq&|S||\Sigma|M(\Phi,\Psi^*)^2\sum_{k\in S}|\scal{a,\Phi_k^*}|^2\\
&\leq&|S||\Sigma|M(\Phi,\Psi^*)^2\beta(\Phi)^2\norm{a}^2.
\end{eqnarray*}
But then
\begin{eqnarray}
\norm{a}^2_{\ell^2(\Psi,\Sigma^c)}&=&\sum_{j=0}^{d-1}|\scal{a,\Psi_j^2}|^2-\sum_{j\in\Sigma}|\scal{a,\Psi_j^*}|^2\nonumber\\
&\geq&\bigl(\alpha(\Psi)^2-|S||\Sigma|M(\Phi,\Psi^*)^2\beta(\Phi)^2\bigr)\norm{a}^2.\label{eq:CS3}
\end{eqnarray}
It then remains to mimic the computation in \eqref{eq:simple} to obtain the result:
write \eqref{eq:CS3} as $\norm{a}\leq D \norm{a}_{\ell^2(\Psi,\Sigma^c)}$ if $\supp_\Phi a\subset S$
(note that the hypothesis on $S,\Sigma$ is equivalent to $D>0$).
Now, if $a\in\C^d$, write $a=a_S+a_{S^c}$ where $\supp_\Phi a_S\subset S$ and $\supp_\Phi a_{S^c}\subset S^c$.
Then
\begin{eqnarray*}
\norm{a}_2&\leq&\|a_S\|_2+\|a_{S^c}\|_2\leq D \|a_S\|_{\ell^2(\Psi,\Sigma^c)}+\|a_{S^c}\|_2\\
&\leq& D \|a\|_{\ell^2(\Psi,\Sigma^c)}+D\|a_{S^c}\|_{\ell^2(\Psi,\Sigma^c)}+\|a_{S^c}\|_2
\leq D \|a\|_{\ell^2(\Psi,\Sigma^c)}+\bigl(1+D\beta(\Psi)\bigr)\|a_{S^c}\|_2\\
&\leq&D \|a\|_{\ell^2(\Psi,\Sigma^c)}+\frac{1+D\beta(\Psi)}{\alpha(\Phi)}\|a\|_{\ell^2(\Phi,S^c)}
\leq \frac{1+D\beta(\Psi)}{\alpha(\Phi)}\bigl(\|a\|_{\ell^2(\Psi,\Sigma^c)}+\|a\|_{\ell^2(\Phi,S^c)}\bigr).
\end{eqnarray*}
Of course, we may exchange the roles of $\Phi$ and $\Psi$ in these computations.
\end{remark}

%\begin{remark}\ \\
%Let $\Phi=\{\Phi_j\}_{j=0,\ldots,d-1}$ and $\Psi=\{\Psi_j\}_{j=0,\ldots,d-1}$ be two orthonormal bases of $\C^d$ and let $U$ be the operator defined by $U\Psi_j=\Phi_j$ for $j=0,\ldots,d-1$.
%
%Note that a slight modification of \eqref{eq:CS1} gives
%$$
%\norm{a}_{\ell^2(\Psi,\Sigma^c)}^2\geq
%\Bigl(1-\norm{P_\Sigma UP_S}_{\ell^2\to\ell^2}^2\Bigr)\norm{a}_{2}
%$$
%if $\supp_\Phi\,a\subset S$. 
%It follows from \eqref{eq:CS2} and \eqref{eq:CS1} that, if
%$s|\Sigma|<\dst\frac{1}{M(\Phi,\Psi)^2}$ then
%$(U,s,\Sigma^c)$ satisfy the uniform uncertainty principle with restriced isometry constant
%$\delta_s\leq M(\Phi,\Psi)^2s|\Sigma|$. As a consequence, Theorem \ref{th:candes} applies
%as soon as $s|\Sigma|\leq\frac{\sqrt{2}-1}{2}M(\Phi,\Psi)^{-2}$. Writing $\Omega=\Sigma^c$, we may deduce from this that, if $s\leq\dst \frac{\sqrt{2}-1}{2(d-|\Omega|)|M(\Phi,\Psi)^2}$, then any $a$ such that
%$\norm{a}_{\ell^0(\Phi)}\leq s$ can be recovered as the unique solution of
%$$
%\min\{\norm{\tilde a}_{\ell^1}\,:\ \scal{\tilde a,\Psi_j}=\scal{a,\Psi_j}\mbox{ for every }j\in\Omega\}.
%$$
%Earlier results in that spirit may be found in \cite{DH,EB,GN}.
%\end{remark}

\medskip

In order to illustrate our main theorem, let us show that a vector
can not be too compressible in two different bases. First, let us recall the definition.

\begin{definition}\ \\
Let $C>0$ and $\alpha>1/2$.
We will say that $a\in\C^d$ is $(C,\alpha)$-compressible in the basis $\Phi$ if, for $j=0,\ldots,d-1$,
the $j$-th biggest coefficient $|\scal{a,\Phi}|^*(j)$
of $a$ in the basis $\Phi$ satisfies $\displaystyle 
|\scal{a,\Phi}|^*(j)\leq \sqrt{2\alpha-1}\frac{C}{(j+1)^\alpha}\|a\|$.
\end{definition}

We will restrict our statement to a simple enough case, the proof being easy to adapt to more general settings:

\begin{corollary}\label{cor:comp}\ \\
Let $\Phi$ and $\Psi$ be two unbiased orthonormal bases of $\C^d$. Let $d\geq 4$, $C>0$ and $\alpha>1/2$ be such that
$\displaystyle C<\frac{( [\sqrt{d}]-1)^{\alpha -\frac{1}{2}}}{4\sqrt{d}}$
(where $[x]$ is the largest integer less than $x$). Then the only vector $a$ that is $(C,\alpha)$-compressible
in both bases is $0$.
\end{corollary}

\begin{proof}
Let $a\not=0$ and assume that $a$ is $(C,\alpha)$-compressible
in both bases. Without loss of generality, we may assume that $\|a\|_2=1$.

Let $\sigma=\sigma_{\Phi}$ be a permutation such that
$\Big(|\langle a, \Phi_{\sigma(j)}\rangle|\Big)_{0\leq j\leq d-1}$ is non-increasing.
For $k=1,\ldots,d$ define $S_{k}= \{\sigma_{\Phi}(0), \ldots, \sigma_{\Phi}(k-1)\}$, the set of the $k$ biggest coefficients of $a$ in the basis $\Phi$. Then
\begin{eqnarray*}
\|a\|_{\ell^{2}(\Phi,S_{k}^{c})}^{2}&=& \sum_{j \notin S_{k} }|\langle a, \Phi_{j}\rangle|^{2}=  \sum_{j=k }^{d-1}|\langle a, \Phi_{\sigma(j)}\rangle|^{2}\\
&\leq&  (2\alpha-1) C^{2} \sum_{j=k+1 }^{d} j^{-2 \alpha}\leq (2\alpha-1) C^{2} \int_{k}^{+\infty} \frac{\mbox{d}x}{x^{2\alpha}}
=\frac{ C^2 }{k^{2\alpha -1}}.
\end{eqnarray*}
It follows that $\|a\|_{\ell^{2}(\Phi,S_{k}^{c})} \leq \frac{ C }{k^{\alpha -\frac{1}{2}}}$. In a similar way, we get
$\|a\|_{\ell^{2}(\Psi,\Sigma_{k}^{c})} \leq \frac{ C }{k^{\alpha -\frac{1}{2}}}$ where $\Sigma_k$ is the set of the $k$ biggest coefficients of $a$ in the basis $\Psi$.

Let us now apply Theorem \ref{th:th1intro} with $S=S_k$ and $\Sigma=\Sigma_k$. Then as long as $k< \sqrt{d}$, 
$\displaystyle 1\leq \frac{2}{1-\frac{k}{\sqrt{d}}} \times \frac{2C}{k^{\alpha -\frac{1}{2}}}$. In other words,
$\displaystyle C \geq \frac{1}{4}\Big( 1-\frac{k}{\sqrt{d}}\Big)k^{\alpha -\frac{1}{2}}$.
 
Assume now that $d\geq 4$ and chose $k=[\sqrt{d}]-1$ so that $k < \sqrt{d}$. It follows that
$$
C \geq\frac{1}{4} \left( 1-\frac{ [\sqrt{d}]-1}{\sqrt{d}}\right)( [\sqrt{d}]-1)^{\alpha -\frac{1}{2}} 
\geq \frac{( [\sqrt{d}]-1)^{\alpha -\frac{1}{2}}}{4\sqrt{d}} 
$$
which completes the proof.
\end{proof}

\begin{remark}\ \\
--- This corollary may be seen as a discrete analogue of Hardy's Uncertainty Principle which states
that an $L^2(\R)$ function and its Fourier transform can not both decrease too fast ({\it see} \cite{HJ,FS}).

\noindent--- The above proof also works if the bases are not unbiased, in which case the condition on $C$ has to be replaced by
$$
C<\frac{M(\Phi,\Psi)}{4}\left(\ent{\frac{1}{M(\Phi,\Psi)}}-1\right)^{\alpha-1/2}.
$$

\noindent--- Let $\Phi$ be an orthonormal basis of $\C^d$ and $a\in\C^d$ with $\norm{a}=1$ and $0< p<2$.
From Bienaym\'e-Chebyshev, we get %that, for $\lambda\geq 0$, 
%$$
%|\{j\,:|\scal{a,\Phi_j}|\geq\lambda\}|\leq \norm{a}_{\ell^p(\Phi)}^p\lambda^{-p}.
%$$
%Applying this to $\lambda=|\scal{a,\Phi}|^*(k)$
%we get 
$$
k+1\leq|\{j\,:|\scal{a,\Phi_j}|\geq|\scal{a,\Phi}|^*(k)\}|\leq \frac{\norm{a}_{\ell^p(\Phi)}^p}{\big(|\scal{a,\Phi}|^*(k)\big)^p}
$$
thus
$$
|\scal{a,\Phi}|^*(k)\leq\frac{\norm{a}_{\ell^p(\Phi)}}{(k+1)^{1/p}}=\sqrt{\frac{2}{p}-1}\left(\sqrt{\frac{p}{p-2}}
\norm{a}_{\ell^p(\Phi)}\right)(k+1)^{-1/p}.
$$
It follows that $a$ is $\dst\left(\sqrt{\frac{p}{2-p}}\norm{a}_{\ell^p(\Phi)},\frac{1}{p}\right)$-compressible in $\Phi$.

This shows that a vector can not have coefficients in two bases with too small $\ell^p$-norm, namely:
{\sl if $\Phi$ and $\Psi$ be two unbiased orthonormal bases of $\C^d$, $d\geq 4$, and if $0<p<2$
then, for every $a\in\C^d$,}
$$
\max\bigl(\norm{a}_{\ell^p(\Phi)},\norm{a}_{\ell^p(\Psi)}\bigr)\geq\sqrt{\frac{2-p}{p}}\frac{([\sqrt{d}]-1)^{\frac{1}{p}-\frac{1}{2}}}{4\sqrt{d}}
\sim \frac{1}{4}\sqrt{\frac{2-p}{p}}d^{\frac{1}{p}-1}.
$$
\end{remark}

\subsection{Results on annihilating pairs using probability techniques}\label{sec:2.3}\ 

So far, we have only used deterministic techniques, which lead to rather weak results.
In this section, we will recall some results that may be obtained using probability methods.

First, let us describe a model of random subsets of average cardinality $k$.
Let $k\leq d$ be an integer and let $\delta_0,\ldots,\delta_{d-1}$ be $d$ independent random variables
take the value $1$ with probability $k/d$ and $0$ with probability $1-k/d$. We then define the \emph{random subset of average cardinality $k$}, $\Omega\subset\{0,\ldots,d-1\}$ by $\Omega=\{i\,:\ \delta_i=1\}$. Those sets
have of course average cardinality $k$ (which is immediate once one write $\mathbf{1}_\Omega=\sum_{j=0}^{d-1}\delta_j\mathbf{1}_j$). 
Moreover, one has the
following standard estimate ({\it see e.g.} \cite[Theorems A.1.12 and A.1.13]{Al} or \cite{Jamem}):
$$
\P\left[|\Omega-k|\ge \frac{k}{2}\right]\leq 2e^{-k/10}.
$$
Therefore, some authors call those sets ``random sets of cardinality $k$''.
In the next section, we will use the following result of
Rudelson-Vershynin \cite{RV}, (improving a result of Cand\`es-Tao): 

\begin{theorem}[Rudelson-Vershynin \cite{RV}]\label{th:RV}\ \\
There exist two absolute constants $C,c$ such that the following holds:
let $\Phi=\{\Phi_0,\ldots,\Phi_{d-1}\}$ and $\Psi=\{\Psi_0,\ldots,\Psi_{d-1}\}$ be two unbiased orthonormal bases of $\C^d$ and let $T\,:\ell^2_d\to\ell^2_d$ be defined by
$T\psi_j=\Phi_j$ for $j=0,\ldots,d-1$.

Let $0<\eta<1$, $t>1$ be real numbers and $s\leq d$ be an integer.
Let $k\geq 1$ be an integer such that,
\begin{equation}
\label{eq:EV}
k\simeq(Cts\log d)\log(Cts\log d)\log^2s.
\end{equation}

Then with probability at least
$1-7e^{-c(1-\eta)t}$, a random set $\Omega$ of average cardinality $k$ satisfies
$$
k-\sqrt{tk}\leq |\Omega|\leq k+\sqrt{tk}
$$
and $(T,\Omega,s)$ satisfies the Uniform Uncertainty Principle with Restricted Isometry Constant
$\delta_s\leq 1-\eta$. In particular, for any $S\subset\{0,\ldots,d\}$ with $|S|\leq s$, for every $a\in\ell^2_d$,
\begin{equation}
\label{eq:RV}
\norm{a}_{\ell^2}\leq \left(1+\sqrt{\frac{d}{\eta|\Omega|}}\right)\bigl(
\norm{a}_{\ell^2(\Phi,S^c)}+\norm{a}_{\ell^2(\Psi,\Omega)} \Bigr).
\end{equation}
\end{theorem}

The parameter $\eta$ is not present in their statement, but it can
be obtained by straightforward modification of their proof.

Taking $s=\dst\frac{d}{\log^5 d}$, $t=\dst\frac{\log d}{2C}$ we obtain $k\simeq d/2$. Thus, with probability $\geq 1-7d^{-\kappa(1-\eta)}$ ($\kappa$ some universal constant) 
$\Omega$ has cardinal $|\Omega|=d/2+O(d^{1/2}\log^{1/2}d)$ and
every set $S$ with cardinal $|S|\leq \frac{d}{\log^5 d}$ and $\Omega^c$ form a strong annihilating pair
in the sense of \eqref{eq:RV} which may now (for $d$ big enough) be reduced to
\begin{equation}
\label{eq:RV2}
\norm{a}_{\ell^2}\leq \frac{2}{\sqrt{\eta}}\bigl(
\norm{a}_{\ell^2(\Phi,S^c)}+\norm{a}_{\ell^2(\Psi,\Omega)} \Bigr).
\end{equation}

\medskip

Another question that one may ask is the following. Given a set $\Sigma$,
does there exist a ``large'' set $S$ such that $(S,\Sigma)$ is an annihilating pair? 
In order to answer this question, let us recall that Bourgain-Tzafriri \cite{BT1} proved the following:

\begin{theorem}[Bourgain-Tzafriri, \cite{BT1}]\footnote{The dimension of the $\ell^2$ space in this theorem is denoted by $n$ as we will apply it to a $n$-dimensional subspace of $\ell^2_d$.}\ \\
If $T\,:\ell^2_n\to\ell^2_n$ is 
such that $\|Te_i\|_2=1$ for $i=0,\ldots,n-1$ (where the $e_i$'s stand for the standard basis of $\ell_n^2$), then there exists a set $\sigma\subset\{0,\ldots,n-1\}$ with $|\sigma|\geq\frac{n}{240\norm{T}_{\ell^2_d\to\ell^2_n}^2}$ such that, for every
$a=(a_j)_{j=0,\ldots,n-1}$ with support in $\sigma$ such that $\norm{Ta}\geq\frac{1}{12}\norm{a}$.

In other words, the matrix of $T$ in the standard basis has a well 
\end{theorem}

In other words, this theorem states that a matrix with columns of norm $1$ has a well-conditioned sub-matrix
of large size.
The original proof of this theorem uses probabilistic techniques (somewhat similar to those used later in \cite{RV}). Recently, an elementary constructive proof of the set $\sigma$ was given by
Spielman-Srivastava \cite{SS}.
The values of the numerical constants where given in \cite{Jamem}.

We may apply this theorem in the following way: consider two mutually unbiased orthonormal bases
$\Phi=\{\phi _j\}$ and $\Psi=\{\psi_j\}$ of $\ell^2_d$ and let $S,\Omega\subset\{0,\ldots,d-1\}$ be two sets with $|S|=|\Omega|=n$ and enumerate them:
$S=\{j_0,\ldots,j_{n-1}\}$ and $\Omega=\{\omega_0,\ldots,\omega_{n-1}\}$. Let $T$ be the operator defined by
$T\phi_{j_k}=\dst\sqrt{\frac{d}{n}}\psi_{\omega_k}$ for $k=0,\ldots,n-1$. Then
$T$ satisfies the hypothesis of Bourgain-Tzafriri's Theorem and $\norm{T}^2\leq\dst\frac{d}{n}$.
Thus there exists $\sigma\subset S$ with $|\sigma|\geq n^2/240d$ such that, for every
$a\in\ell^2_n$ with support in $\sigma$,
$$
\norm{a}_{\ell^2(\Psi,\Omega)}\geq\frac{1}{12}\sqrt{\frac{n}{d}}\norm{a}_{\ell^2(\sigma)}.
$$
From which we immediately deduce the following (where $\Sigma=\Omega^c$):

\begin{proposition}\label{prop:BT}\ \\
Let $\Phi$ and $\Psi$ be two mutually unbiased bases of $\ell^2_d$ and let $S,\Sigma\subset\{0,\ldots,d-1\}$ be two sets with $|S|+|\Sigma|=d$. Then there exists
$\sigma\subset S$ such that $|\sigma|\geq\dst\frac{(d-|\Sigma|)^2}{240d}$ and,
for every $a\in\ell^2_d$,
\begin{equation}
\label{eq:BT2}
\norm{a}_{\ell_2}\leq\frac{13}{\sqrt{1-|\Sigma|/d}}\bigl(\norm{a}_{\ell^2(\Phi,\sigma^c)}
+\norm{a}_{\ell^2(\Psi,\Sigma^c)}\bigr).
\end{equation}
\end{proposition}
Of course, this proposition only makes sense when $|\Sigma|\leq d-\sqrt{240d}$ otherwise there is no guarantee to have $\sigma\not=\emptyset$. We may thus rewrite \eqref{eq:BT2} as
$$
\norm{a}_{\ell_2}\leq 4d^{1/4}\bigl(\norm{a}_{\ell^2(\Phi,\sigma^c)}
+\norm{a}_{\ell^2(\Psi,\Sigma^c)}\bigr).
$$

\section{The uncertainty principle for the discrete short-time Fourier transform}

The short-time Fourier transform (or windowed Fourier transform) is a useful tool in time-frequency analysis and in signal processing. For $f,g\in L^2(\R)$, we define $V_gf$ on $\R^2$ by the formula
$$
V_gf(x,\xi)=\int f(t)\overline{g(t-x)}e^{-2i\pi t\xi}\,\mbox{d}\xi.
$$
This may be rewritten as $V_gf(x,\xi)=\ff[f\overline{\tau_xg}](\xi)$ where $\ff$ is the Fourier transform on
$L^2(\R)$ and $\tau_xg(t)=g(t-x)$ is the translation operator. Written like this, it is straightforward to
generalize this transform to the more general setting of locally Abelian groups $G$ and its dual $\hat G$ as
$$
V_gf(x,\xi)=\int_G f(t)\overline{g(t-x)}\overline{\scal{\xi,t}}\,\mbox{d}\nu(t),\qquad (x,\xi)\in G\times\hat G
$$
where $\mbox{d}\nu_G$ is the Haar measure on $G$.

For the reader that is not acquainted with this general setting, it may be sufficient to consider $G=\Z/d\Z$ the 
cyclic group seen as $\{0,\ldots,d-1\}$,
$\hat G$ the (multiplicative group of) $d$-th roots of unity. Note that, if we identify the $d$-th root of unity
$e^{2i\pi k/d}$ with the integer $k$, then 
$\hat G$ is identified to $\{0,\ldots,d-1\}$.

For $j\in G$ and $\xi\in\hat G$, we write
$\scal{\xi,j}=\xi^j$. Then $L^2(G)$ may be seen either as the set of $d$-periodic sequences or as $\ell^2_d$.
If $a\in L^2(G)$, the discrete Fourier transform $\dst\ff_d[a](k)=\frac{1}{\sqrt{d}}\sum_{j=0}^{d-1}a_je^{-2i\pi jk/d}$ may be seen as a function on $\hat G$ if we identify $k$ with $\zeta=e^{2i\pi k/d}$:
$$
\ff_G[a](\zeta)=\frac{1}{\sqrt{d}}\sum_{j\in\{0,\ldots,d-1\}}a_j\overline{\scal{\zeta,j}}.
$$
Note that $G$ may also be seen as the ``dual group'' of $\hat G$ if we write $\scal{j,\xi}=\overline{\scal{\xi,j}}$
for $j\in G$ and $\xi\in\hat G$. The Fourier transform on $\hat G$ is then defined by
$$
\ff_{\hat G}[b](j)=
\frac{1}{\sqrt{d}}\sum_{\zeta\in\{1,e^{2i\pi/d},\ldots,e^{2i\pi(d-1)/d}\}}b_\zeta\overline{\scal{j,\zeta}}
=\frac{1}{\sqrt{d}}\sum_{k\in\{0,\ldots,d-1\}}b_ke^{2i\pi kj/d}
$$
if we write $b_k$ for $b_{e^{2i\pi k/d}}$ ({\it i.e.} if we identify the $d$-th roots of unity with 
$\{0,\ldots,d-1\}$). Thus the Fourier transform on $\hat G$ is the inverse discrete Fourier transform.
Finally, we will use the Fourier transform on $G\times\hat G$, this is then just the discrete Fourier transform in the first variable and the inverse discrete Fourier transform in the second one.

Now take $f,g$ two $d$-periodic sequences. With these notations, the discrete short-time Fourier transform is
defined on $G\times\hat G=\{0,\ldots,d-1\}\times\{0,\ldots,d-1\}$ as
$$
V_gf(j,k)=\frac{1}{\sqrt{d}}\sum_{\ell=0}^{d-1}f_\ell \overline{g_{\ell-j}}e^{-2i\pi k\ell/d}.
$$
The symmetry lemma (Lemma \ref{lem:sym}) then reads
$$
\ff_d\otimes\ff_d^{-1}[V_gf\overline{V_hk}](j,k)=V_kf(-k,j)\overline{V_hg(-k,j)}.
$$

The reader that does not want to enter the details concerning finite Abelian groups nor the proof of the symmetry lemma may now skip the following two sections and replace $G$, $\hat G$
by $\{0,\ldots,d-1\}$ in the statements of Section \ref{sec:3.3}.

\subsection{Finite Abelian groups}
\label{sec:abel}\ \\
In this section, we recall some notations on the Fourier transform on finite Abelian groups.
Results stated here may be found in \cite{Te} and (with slightly modified notations) in \cite{KPR}.

Throughout the remaining of this paper, we will denote by $G$ a finite Abelian group for which the group law will be denoted additively. The identity element of $G$ is denoted by
$0$. The dual group of characters $\hat G$
of $G$ is the set of homomorphisms $\xi\in\hat G$ which map $G$ into the multiplicative group
$\S^1 = \{z \in\C\,:\ |z| = 1\}$. The set $\hat G$ is an Abelian group under pointwise multiplication and, as is
customary, we shall write this commutative group operation additively. Note that $G$ is isomorphic
to $\hat G$, in particular $|G|=|\hat G|$. Further, Pontryagin duality implies that
$\widehat{\hat G}$ can be canonically identified with $G$, a fact which
is emphasized by writing $\scal{\xi,x}=\xi(x)$. Note that, as group operations are written additively,
$$
\scal{-\xi,x}=\scal{\xi,-x}=\overline{\scal{\xi,x}}.
$$

The Fourier transform $\ff_G f =\hat f \in \C^{\hat G}$ of $f\in\C^G$ is given by
$$
\hat f(\xi)=\frac{1}{|G|^{1/2}}\sum_{x\in G}f(x)\overline{\scal{\xi,x}},\qquad\xi\in\hat G.
$$
The transform is unitary\,: $\|\hat f\|_2=\|f\|_2$, thus $\ff_G$ is invertible. The inversion formula is given
by the following
$$
f(x)=\ff_{\widehat{G}}[\overline{\hat f}\,](x)=\frac{1}{|G|^{1/2}}\sum_{\xi\in\hat G}\hat f(\xi)\scal{\xi,x},\qquad x\in G.
$$
Moreover, as the normalized characters
$\{|G|^{-1/2}\xi\}_{\xi\in\hat G}$ form an orthonormal basis of $\C^G$ that is unbiased with the standard basis
we can reformulate Theorem \ref{th:th1intro} as follows:

\medskip
\noindent{\bf Strong Uncertainty Principle on Finite Abelian Groups.}\\
{\sl Let $G$ be a finite Abelian group and let
$S\subset G$ and $\Sigma\subset\hat G$ be such that $|S||\Sigma|<|G|$. Then for every $f\in\C^G$,}
\begin{eqnarray}
\label{eq:SUP}
\|f\|_2\leq\frac{2}{1-(|S||\Sigma|/|G|)^{1/2}}\left[\left(\sum_{x\notin S}|f(x)|^2\right)^{1/2}+
\left(\sum_{\xi\notin \Sigma}|\hat f(\xi)|^2\right)^{1/2}\right].
\end{eqnarray}

\medskip

%The corresponding ``weak'' Uncertainty Principle appeared for the first time in \cite{MS} and has been
%re-discovered several times, including \cite{DS}.
%Note also that one may improve this result by choosing the sets $S,\Sigma$ randomly
%(\cite{Tr} for the cyclic groups). Moreover, when $G=\Z/p\Z$, $p$ a prime,
%Tao \cite{Ta} showed that the condition $|S||\Sigma|<|G|$ could be relaxed to $|S|+|\Sigma|\leq|G|$
%({\it see also} \cite{Me}). One then obtains an inequality of the form
%$\|f\|_2\leq C(S,\Sigma)\left[\left(\sum_{x\notin S}|f(x)|^2\right)^{1/2}+
%\left(\sum_{\xi\notin \Sigma}|\hat f(\xi)|^2\right)^{1/2}\right]$
%but no estimate on $C(S,\Sigma)$ seems to be known. Actually, Tao only proved the weak formulation
%of this result, the constant $C(S,\Sigma)$ is thus obtained by the compactness argument given above.
%
For any $x\in G$, we define the translation operator $T_x$ as the unitary operator on $\C^G$ given by
$T_xf(y) = f(y-x)$, $y\in G$. Similarly, we define the modulation operator $M_\xi$ for $\xi\in\hat G$ as the unitary
operator defined by $M_\xi f = f\cdot \xi$, where here and in the following $f\cdot g$ denotes the pointwise product
of $f, g \in\C^G$. Since $\widehat{M_\xi f} = T_\xi\hat f$, we refer to $M_\xi$ also as a frequency shift operator.
Note also that $\widehat{T_x f} = M_{-x}\hat f$

We denote by $\pi(\lambda) = M_\xi T_x$, $\lambda= (x,\xi) \in G\times\hat G$ the \emph{time-frequency shift operators}.
Note that these are unitary operators. The \emph{short-time Fourier transformation} $V_g^G\,: C^G \to\C^{G\times\hat G}$
with respect to the window $g\in\C^G\setminus\{0\}$ is given for $x\in G$, $\xi\in\hat G$ by
$$
V_g^Gf(x,\xi)=\frac{1}{|G|^{1/2}}\scal{f,\pi(x,\xi)g}
=\frac{1}{|G|^{1/2}}\sum_{y\in G}f(y)\overline{g(y-x)}\,\overline{\scal{\xi,y}}
=\ff_G[f\cdot\overline{T_x g}](\xi)
$$
where $f\in\C^G$. The inversion formula for the short-time Fourier transform is
$$
f(y) = \frac{1}{|G|^{1/2}\|g\|_2^2}\sum_{(x,\xi)\in G\times\hat G}V_g^Gf(x,\xi)g(y-x)\scal{\xi,y}.
$$
Further, 
$\|V_gf\|_2=\|f\|_2\|g\|_2$, in particular $V_gf=0$ if and only if either $f=0$ or $g=0$.

Finally, let us note that a simple computation shows that
\begin{equation}
\label{eq:tfshif}
V^G_{\pi(b,v)g}\pi(a,u)f(x,\xi)=\scal{u-v-\xi,a}\scal{v,x}V_g^Gf(x-a+b,\xi-u+v).
\end{equation}

\subsection{The symmetry lemma}\ \\
Let us first note that the short-time Fourier transform on $\hat G$ is defined by
$$
V^{\hat G}_\gamma\ffi(\xi,x)=\frac{1}{|\hat G|^{1/2}}\scal{\ffi,M_xT_\xi\gamma}
=\frac{1}{|\hat G|^{1/2}}\sum_{\eta\in\hat G}\ffi(\eta)\overline{\gamma(\eta-\xi)}\,\overline{\eta(x)}.
$$
This is linked to $V^G$ in the following way:
\begin{eqnarray*}
V_g^Gf(x,\xi)&=&\frac{1}{|G|^{1/2}}\scal{f,\pi(x,\xi) g}=\frac{1}{|G|^{1/2}}\scal{\ff^Gf,\ff^G[M_\xi T_x g]}\\
&=&\frac{1}{|\hat G|^{1/2}}\scal{\ff^Gf,T_\xi M_{-x}\ff^G g}
=\frac{\overline{\scal{\xi,x}}}{|\hat G|^{1/2}}\scal{\ff^Gf,M_{-x}T_\xi  \ff^Gg},
\end{eqnarray*}
so that
\begin{equation}
\label{eq:fundstft}
V_g^Gf(x,\xi)=\overline{\scal{\xi,x}}V_{\hat g}^{\hat G}\hat f (\xi,-x).
\end{equation}

\begin{lemma} 
\label{lem:sym}
Let $f,g,h,k\in\C^G$. Then for every $u\in G$ and every $\eta\in\hat G$,
$$
\ff_{G\times\hat G}[V_g^{G}f\,\overline{V_h^Gk}](\eta,u)=V_k^Gf(-u,\eta)\overline{V_h^Gg(-u,\eta)}.
$$
\end{lemma}

\begin{proof}
First note that
\begin{eqnarray*}
\ff_{G\times\hat G}[V_g^{G}f\,\overline{V_h^Gk}](\eta,u)
&=&\frac{1}{|G|^{1/2}|\hat G|^{1/2}}\sum_{x\in G}\sum_{\xi\in\hat G}
V_g^{G}f(x,\xi)\,\overline{V_h^Gk(x,\xi)}\,\overline{\scal{\eta,x}}\,\overline{\scal{\xi,u}}\\
&=&\frac{1}{|G|^{1/2}|\hat G|^{1/2}}\sum_{x\in G}\sum_{\xi\in\hat G}
V_g^{G}f(x,\xi)\,\overline{V_{\hat h}^{\hat G}\hat k(\xi,-x)}\,\overline{\scal{\eta,x}}\,\overline{\scal{\xi,u-x}}
\end{eqnarray*}
with \eqref{eq:fundstft}. Using the definition of the short-time Fourier transform, this is further equal to
\begin{eqnarray}
&&\frac{1}{|\hat G||G|}\sum_{x\in G}\sum_{\xi\in\hat G}\sum_{y\in G}\sum_{\zeta\in\hat G}
f(y)\overline{g(y-x)}\,\overline{\scal{\xi,y}}\,\overline{\hat k(\zeta)}
\hat h(\zeta-\xi)\scal{\zeta,-x}\,\overline{\scal{\eta,x}}\,\overline{\scal{\xi,u+x}}\notag\\
&&\qquad\qquad=\frac{1}{|\hat G||G|}\sum_{x\in G}\sum_{\xi\in\hat G}\sum_{y\in G}\sum_{\zeta\in\hat G}
f(y)\overline{g(y-x)}\,\overline{\hat k(\zeta)}
\hat h(\zeta-\xi)\,\overline{\scal{\eta+\zeta,x}}\,\overline{\scal{\xi,u-x+y}}\label{eq:aaa}
\end{eqnarray}
We will now invert the orders of summation. First
\begin{eqnarray*}
&&\frac{1}{|\hat G|^{1/2}}\sum_{\xi\in\hat G}\hat h(\zeta-\xi)\overline{\scal{\xi,u-x+y}}
=\frac{1}{|\hat G|^{1/2}}\sum_{\xi\in\hat G}\hat h(\xi+\zeta)\scal{\xi,u-x+y}\\
&&\hspace{3cm}=\frac{1}{|\hat G|^{1/2}}\sum_{\xi\in\hat G}\widehat{M_{-\zeta} h}(\xi)\scal{\xi,u-x+y}
=[M_{-\zeta }h](u-x+y)\\
&&\hspace{3cm}=\scal{\zeta,x}\scal{-\zeta,u+y}h(u-x+y).
\end{eqnarray*}
Then
\begin{eqnarray*}
&&\frac{1}{|\hat G|^{1/2}|G|^{1/2}}\sum_{x\in G}\overline{g(y-x)}\,\overline{\scal{\zeta+\eta,x}}
\sum_{\xi\in\hat G}\hat h(\zeta-\xi)\overline{\scal{\xi,u-x+y}}\\
&&\hspace{3cm}=\frac{\overline{\scal{\zeta,u+y}}}{|G|^{1/2}}\sum_{x\in G}\,\,\overline{g(y-x)}\,\overline{\scal{\eta,x}}
h(u-x+y)\\
&&\hspace{3cm}=\frac{\overline{\scal{\zeta,u+y}}}{|G|^{1/2}}\sum_{z\in G}\overline{g(z)}h(z+u)\scal{\eta,z}
\,\overline{\scal{\eta,y}}\\
&&\hspace{3cm}=\overline{\scal{\zeta,u}}\,\overline{\scal{\eta+\zeta,y}}\,\overline{V_h^Gg(-u,\eta)}.
\end{eqnarray*}
It follows that
\begin{eqnarray*}
&&\frac{1}{|\hat G||G|^{1/2}}\sum_{\zeta\in\hat G}\overline{\hat k(\zeta)}
\sum_{x\in G}\overline{g(y-x)}\,\overline{\scal{\zeta+\eta,x}}
\sum_{\xi\in\hat G}\hat h(\zeta-\xi)\overline{\scal{\xi,u-x+y}}\\
&&\hspace{3cm}=\frac{1}{|\hat G|^{1/2}}\sum_{\zeta\in\hat G}\overline{\hat k(\zeta)}
\overline{\scal{\zeta,u}}\,\overline{\scal{\eta+\zeta,y}}\,\overline{V_h^Gg(-u,\eta)}\\
&&\hspace{3cm}=\frac{\overline{\scal{\eta,y}}\,\overline{V_h^Gg(-u,\eta)}}{|\hat G|^{1/2}}
\sum_{\zeta\in\hat G}\overline{\hat k(\zeta)}\overline{\scal{\zeta,y+u}}\\
&&\hspace{3cm}=\overline{\scal{\eta,y}}\,\overline{V_h^Gg(-u,\eta)}\,\overline{k(y+u)}.
\end{eqnarray*}
Finally, it remains to take the sum in the $y$-variable in \eqref{eq:aaa} to obtain
$$
\ff_{G\times\hat G}[V_g^{G}f\,\overline{V_h^Gk}](\eta,u)=V_k^Gf(-u,\eta)\overline{V_h^Gg(-u,\eta)}.
$$
as announced.
\end{proof}

In the case $G=\R^d$, this lemma was given independently in \cite{Jam,Jan}.

\subsection{The Uncertainty Principle for the short-time Fourier transform}\label{sec:3.3}\ \\
We will conclude this section with the following lemma that allows to transfer 
results about strong annihilating pairs in $G\times \hat G$ to
Uncertainty Principles for the short-time Fourier transform and then give two corollaries.

\begin{lemma}\label{lem:transfer}
Let $\Sigma\subset G\times\hat G$ and $\tilde\Sigma=\{(\xi,-x)\,:\ (x,\xi)\in\Sigma\}\subset\widehat{G\times\hat G}=\hat G\times G$.
Assume that $(\Sigma,\tilde\Sigma)$ is a strong annihilating pair in $G\times\hat G$, {\it i.e.}
that there is a constant $C(\Sigma)$ such that, for every $F\in\C^{G\times\hat G}$,
$$
\|F\|_2^2\leq C(\Sigma)\left(\sum_{(x,\xi)\notin\Sigma}|F(x,\xi)|^2
+\sum_{(x,\xi)\notin\tilde\Sigma}|\ff_{G\times\hat G}F(\xi,x)|^2\right)
$$
then for every $f,g\in\C^G$, with $\|g\|_2=1$,
$$
\|f\|_2^2\leq 2C(\Sigma)\sum_{(x,\xi)\notin\Sigma}|V_g^Gf(x,\xi)|^2.
$$
\end{lemma}

\begin{proof} We will adapt the proof in the case $G=\R^d$ given in \cite{De} to our situation.
Let us fix  $f,g\in\C^G$.

We will only use Lemma \ref{lem:sym} in a simple form: for
$a\in G,\eta\in\hat G$ define the function $F_{a,\eta}$
on $G\times\hat G$ by
$$
F_{a,\eta}(x,\xi)=\overline{\scal{\xi-\eta,a}}V_f^Gg(x-a,\xi-\eta)\overline{V_g^Gf(x,\xi)}.
$$
Note that $F_{a,\eta}(x,\xi)=V_f^G\pi(a,\eta)g\overline{V_{\pi(a,\eta)g}^Gf}$ so that
then $\ff_{G\times\hat G}F_{a,\eta}(\xi,x)=F_{a,\eta}(-x,\xi)$.

It follows that
\begin{eqnarray}
\|F_{a,\eta}\|_2^2&\leq& C(\Sigma)\left(\sum_{(x,\xi)\notin\Sigma}|F_{a,\eta}(x,\xi)|^2
+\sum_{(x,\xi)\notin\tilde\Sigma}|F_{a,\eta}(-x,\xi)|^2\right)\notag\\
&=&2C(\Sigma)\sum_{(x,\xi)\notin\Sigma}|V_f^Gg(x-a,\xi-\eta)|^2|V_g^Gf(x,\xi)|^2.\label{eq:aaabbb}
\end{eqnarray}
Now note that
\begin{eqnarray*}
\sum_{(a,\eta)\in G\times\hat G}\|F_{a,\eta}\|_2^2&=&
\sum_{(a,\eta)\in G\times\hat G}\sum_{(x,\xi)\in G\times\hat G}|V_f^Gg(x-a,\xi-\eta)|^2|V_g^Gf(x,\xi)|\\
&=&\|V_f^Gg\|_2\|V_g^Gf\|_2=\|f\|_2^4\|g\|_2^4
\end{eqnarray*}
where we inverted the summation over $(a,\eta)$ and the summation over $(x,\xi)$.
Finally, summing inequality \eqref{eq:aaabbb} over $(a,\eta)\in G\times\hat G$ gives
$$
\|f\|_2^4\|g\|_2^4
\leq 2C(\Sigma)\|f\|_2^2\|g\|_2^2\sum_{(x,\xi)\notin\Sigma}|V_g^Gf(x,\xi)|^2
$$
which completes the proof.
\end{proof}

Combining this result with  \eqref{eq:SUP} we immediately get the following:

\begin{corollary}\label{cor:cor2intro}\ \\
Let $\Sigma\subset G\times\hat G$ with $|\Sigma|< |G|$. Let $g\in\C^G$ with $\|g\|_2=1$.
Then for every $f\in\C^G$,
$$
\|f\|_2^2\leq \frac{8}{(1-|\Sigma|/|G|)^2}\sum_{(x,\xi)\notin\Sigma}|V_g^Gf(x,\xi)|^2.
$$
\end{corollary}

The corresponding weak annihilating property for $\Sigma$ was obtained by F. Krahmer, G.\,E. Pfander and P. Rashkov
\cite{KPR}.

Finally, using the fact that two random events $A$ and $B$ that each occur with probability $\geq 1-\alpha$,
jointly occur with probability $\geq 1-2\alpha$, we deduce the following from
Theorem \ref{th:RV}, as reformulated in \eqref{eq:RV2}:

\begin{corollary}\label{cor:RVamb}\ \\
There exist two absolute constants $C,c$ such that the following holds:
Let $0<\eta<1$, let $g\in\ell^2_d$ with $\norm{g}_2=1$.

Then with probability at least
$1-14|G|^{-c(1-\eta)}$, a random set $\Omega$ of average cardinality $|G|/2$ satisfies
$$
|\Omega|= |G|/2+O(|G|^{1/2}\log^{1/2}|G|)
$$
and, for any $S\subset G$ with $|S|\leq \frac{|G|}{\log^5 |G|}$, for every $f\in\C^G$,
\begin{eqnarray}
\label{eq:RVstft}
\norm{f}_{\ell^2}&\leq& \frac{2\sqrt{2}}{\sqrt{\eta}}
\left(\sum_{x\notin S, \xi\in\Omega}|V_g^Gf(x,\xi)|^2\right)^{1/2}.
\end{eqnarray}
\end{corollary}
%
%It would therefore be nice to have a ``quadratic'' analogue of Cand\`es's Theorem \ref{th:candes}
%in order to obtain reconstruction of $f$ and $g$ from lacunary data $\{V_f^Gg(x,\xi),\ x\in S^c,\ \xi\in\Omega\}$. 

\section{Conclusion and future directions}

In this paper, we have shown how to obtain quantitative uncertainty principles for the
representation of a vector in two different bases. These estimates are stated in terms
of annihilating pairs and both extend and simplify previously known qualitative results.
We then apply our main theorem to the discrete short time Fourier transform,
following the path of corresponding results in the continuous setting.

\medskip

Let us now describe a question raised by our work. First note that we may rewrite \eqref{eq:RVstft}
\begin{equation}
\norm{f}_2\leq\frac{2\sqrt{2}}{\sqrt{\eta}}\frac{1}{|G|^{1/2}}\left(\sum_{x\notin S, \xi\in\Omega}|\scal{f,\pi(x,\xi)g}|^2\right)^{1/2}.
\label{eq:RVstft2}
\end{equation}
But, the family $\{\pi(x,\xi)g,\ x\in G,\ \xi\in\hat G\}$ forms a so-called
finite (tight) Gabor frame ({\it see} \cite{KPR} and references therein for more on finite
Gabor frames). In other words, we have a system of $|G||\hat G|=|G|^2:=d^2$ vectors
$e_1,\ldots,e_{d^2}$ in $\C^d$ such that, for every $f\in\C^d$
$$
\sum_{j=1}^{d^2}|\scal{f,e_j}|^2=d\norm{f}^2.
$$
If we write $\tilde\Omega$ for the subset of $\{1,\ldots,d^2\}$ such that
$\{e_j,j\in\tilde\Omega\}$ is an enumeration of  $\{\pi(x,\xi)g,\ x\notin S,\ \xi\in\Omega\}$, then
\eqref{eq:RVstft2} may be rewritten as
$$
\frac{\eta d}{8}\norm{f}^2\leq \sum_{j\in\tilde\Omega} |\scal{f,e_j}|^2
$$
As a consequence, we obtain that, if $\scal{f_1,e_j}=\scal{f_2,e_j}$ for every $j\in\tilde\Omega$, then applying this to $f=f_1-f_2$, we obtain that $f_1=f_2$. It would thus be desirable to have an algorithm that allows to
reconstruct $f$ from its frame coefficients $\{\scal{f,e_j},j\in\tilde\Omega\}$. In particular, we ask the following:

\begin{problem}
Let $\{e_1,\ldots,e_{d^2}\}$ be a Gabor frame in $\C^d$.
Assume that $\tilde\Omega\subset\{1,\ldots,d^2\}$ and $0<\delta<1$ are such that, for every $f\in\C^d$,
$$
\delta d\norm{f}^2\leq \sum_{j\in\tilde\Omega} |\scal{f,e_j}|^2.
$$
Is it true that every $f$ in $\C^d$ is given by
$$
f=\mathrm{argmin}\left\{\sum_{j=1}^{d^2} |\langle\tilde f,e_j\rangle|\,:\langle\tilde f,e_j\rangle=\scal{f,e_j}\ \forall j\in\tilde\Omega\right\}?
$$
\end{problem}

Note that no sparsity is assumed on $f$ here which make this problem differ from the one considered by
Pfander-Rauhut \cite{PR}. Also, we expect that there is a minimal
$\delta_0>0$ such that this property only holds for $\delta_0<\delta<1$.

%But, according to \cite[Proposition 1.1]{HL}, if we denote by $\{\delta_1,\ldots,\delta_{d^2}\}$ the
%standard orthonormal basis of $\C^{d^2}$ and an orthogonal projection $P$ from $\C^{d^2}$ onto $\C^d$ such that
%$f_j=P\delta_j$. 

A slightly different problem that may arise in radar theory is that of recovering $f$ \emph{and} $g$
from partial knowledge of $V_g^Gf$. It is totally unclear to us whether our results may contribute to this
task. In particular, note that this problem is quadratic (bilinear), so that an identity such as \eqref{eq:RVstft}-\eqref{eq:RVstft2} does not
immediately imply that if 
$V_g^Gf=V_{\tilde g}^G\tilde f$ on $S^c\times\Omega$, then $\tilde g=cg$ and $\tilde f=cf$
with $c\in\C$, $|c|=1$.

\subsection*{Ackowledgments}  Both authors were partially supported by the
ANR project AHPI {\it Analyse Harmonique et Probl\`emes Inverses}.

The authors are also particularly grateful to the anonymous referee who helped in clarifying the presentation of
the manuscript and pointed out a normalization issue in the UUP.

\end{document}